\documentclass{cccg23}
\usepackage{graphicx, amssymb, amsmath} % Required for inserting images

\usepackage{algorithm}
\usepackage[noend]{algpseudocode}

\usepackage{mathtools}
\usepackage{makecell}

\newcommand{\myheader}[1]{\smallskip\noindent\textbf{#1}}

\newtheorem{property}{Property}
\title{Computing Representatives of Persistent Homology Generators with a Double Twist}
\author{Tuyen Pham\thanks{University of Florida, Gainesville; {\tt tuyen.pham@ufl.edu}}, Hubert Wagner \thanks{University of Florida, Gainesville; {\tt hwagner@ufl.edu}}}
\date{}
\index{Wagner, Hubert}

\begin{document}
\thispagestyle{empty}
\maketitle

\begin{abstract}
With the growing availability of efficient tools, persistent homology is becoming a useful methodology in a variety of applications. Significant work has been devoted to implementing tools for persistent homology diagrams; however, computing representative cycles corresponding to each point in the diagram can still be inefficient. To circumvent this problem, we extend the twist algorithm of Chen and Kerber. Our extension is based on a new technique we call \emph{saving}, which supplements their existing \emph{killing} technique. The resulting two-pass strategy can be realized using an existing matrix reduction implementation as a black-box and improves the efficiency of computing representatives of persistent homology generators. We prove the correctness of the new approach and experimentally show its performance.
\end{abstract}

\section{Overview}
Persistent homology is a popular methodology for studying geometric and topological information of data. Briefly, as we vary a parameter, persistent homology captures the creation and destruction of topological features present in the data. Typically a persistence diagram is used as a concise geometric-topological descriptor of this evolution. Its usage is becoming popular in various applied fields, including medical imaging~\cite{hu2019topology}, astronomy~\cite{heydenreich2021persistent}, genetics~\cite{rabadan2020identification} and material science~\cite{hiraoka2016hierarchical}.

In many applications it is useful to go beyond this standard descriptor, and study a geometric representation of the captured topological features. More technically, we are referring to the geometry of the representatives of persistent homology generators, which we explain in the next section along with other technicalities; see Figure~\ref{fig:VRex} for an illustration. Visualizing topological information can make topological analysis more intuitive and transparent. Indeed, our work is motivated by a recent project on analyzing neural networks using persistent homology~\cite{zheng2021topological}. During this project, we encountered some computational obstacles related to computing representative cycles -- and overcome them by creative use of available tools. This enabled visualization and further analysis of important topological features in the data. We share the developed techniques, as they can be applied more generally.

\begin{figure}
    \centering
    \includegraphics[width=.45\textwidth]{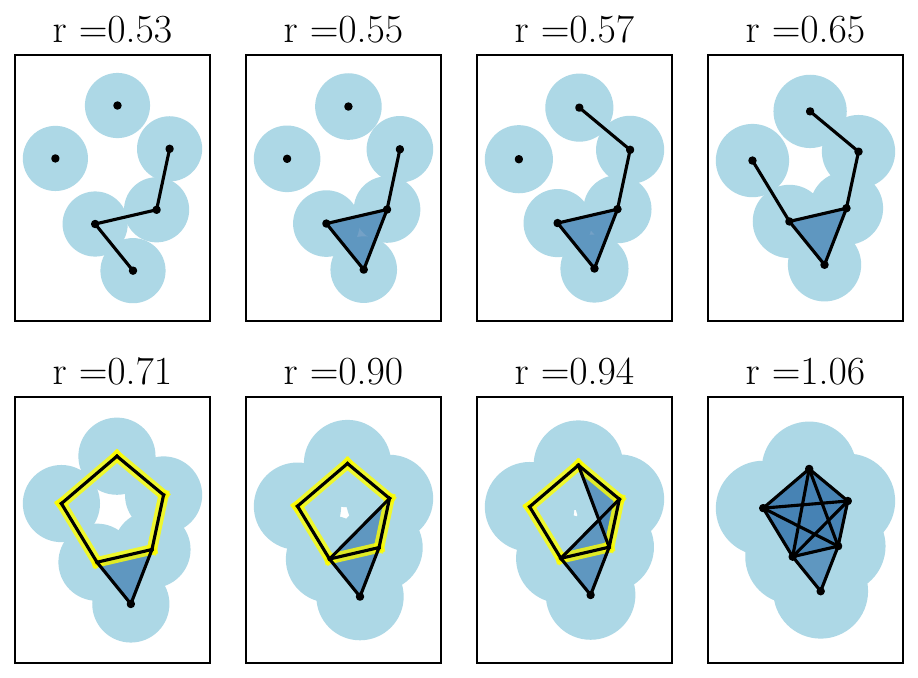}
    \caption{Example of a Vietoris--Rips filtration which approximates the topology of the growing union of disks. Typically information about the birth and death of topological features would be encoded as a persistence diagram. The cycle highlighted in yellow (born at radius 0.71, and destroyed at 1.06) is one representative cycle returned by the algorithm we propose.}
    \label{fig:VRex}
\end{figure}

These techniques are beneficial because most existing software packages focus on optimizing the computation of persistence diagrams, and not the cycle representatives. Interestingly, the standard algorithm for persistent homology produces cycle representatives (of non-essential classes) with no extra work. On the other hand, applying certain crucial optimizations complicates the situation. As a result, persistence diagrams are typically computed in time approximately linear in the number of input simplices -- but computing the representatives can scale quadratically in practical situations. In this work, we aim to close this performance gap by proposing an efficient algorithm for computing representatives of persistent homology generators.

One setting in which our results are particularly useful is low-dimensional skeleta of geometric complexes describing high-dimensional point clouds. In this scenario, it is beneficial~\cite{phat} to switch to persistent cohomology~\cite{deSilva_2011} and apply the crucial killing optimization by Chen and Kerber~\cite{Chen2011twist}. While a duality between homology and cohomology allows us to efficiently compute the persistent homology diagram, we directly obtain only representatives of persistent cohomology generators -- and not their easier to interpret and visualize homological counterparts. We offer a simple computational technique which allows us to obtain representatives of persistent homology generators with little extra overhead while benefiting from these crucial optimizations.

\myheader{Contributions.} In short, we propose a more optimistic counterpart of the \emph{killing} technique which we call the \emph{saving} technique. It allows us to generate only a subset of the columns of the boundary matrix, without affecting the results. This technique is part of a new two-pass strategy for computing representatives of non-essential persistent homology classes. In the first pass, we reduce the coboundary matrix using the usual twist algorithm, and generate a subset of the boundary matrix. In the second pass, we reduce the pruned boundary matrix which allows us to retrieve the representatives. We implement this strategy and experimentally show it is typically much faster than reducing the original boundary matrix.

\myheader{Structure of the paper.} In Section~\ref{sec:background}, we review the usual mathematical background related to persistent homology, trying to make it accessible to audiences with limited exposure to algebraic topology. In Section~\ref{sec:related} we briefly review literature on computational aspects of persistent homology. In Section~\ref{sec:algs} we explain in more details the techniques and algorithms we use in our approach. In Section~\ref{sec:doubletwist} we explain our approach and then experimentally show its efficiency in Section~\ref{sec:experiments}. We conclude the paper in Section~\ref{sec:summary}.

%We focus on three aspects: (1) simple explanation our method, including why it and when it is useful; (2) proof of its correctness; (3) experiments showing the speedup compared to the existing, highly-optimized matrix reduction implementations. Additionally, we think that reiterating the importance of the existing killing (clearing) techniques is valuable.

\section{Mathematical background}
\label{sec:background}
We offer a quick review of the algebraic machinery behind persistent homology~\cite{edelsbrunner2010computational}. We focus on its common usage in which one computes a sequence of simplicial complexes describing the geometry and topology of a finite point set in $\mathbb{R}^D$. One popular choice is the Vietoris--Rips construction. It allows us to track the birth and death of topological features as a scale parameter is varied. Our main focus is on representatives of persistent homology, which additionally allow us to find geometric representation of topological features. 

A $k$-simplex is the convex hull of $k+1$ affinely independent points in $\mathbb{R}^D$. For $k = 0,1,2,3$, these are vertices, edges, trianges and tetrahedra. A face of a simplex is the convex hull of a subset of its vertices. A face of a simplex $\sigma$ is called proper if its dimension is one less than the dimension of $\sigma$. The boundary of a simplex is the set of its proper faces. By a simplicial complex $K$ we mean collection of simplices such that for every simplex $\sigma\in K$, every face of $\sigma$ is also in $K$. 

Before we discuss homology groups, we define $k$-chains as formal sums of simplices with coefficients in $\mathbb{Z}_2$. The $k^{th}$ chain group, $C_k(K)$ is formed by $k$-chains along with elementwise addition. These chains can be viewed as subsets of simplices in $K$; the addition reduces to exclusive-difference operation. We define boundary homomorphisms $\partial_{k}:C_{k}\to C_{k-1}$ mapping a chain to the sum of the boundaries of the simplices with nonzero coefficients. Crucially, taking the boundary of any chain twice yields the 0 chain. Because of this property, we can define $k$-\textit{cycles} in $C_k(K)$ as $\ker\partial_n$ and $k$-\emph{boundaries} in $C_k(K)$ as $\operatorname{im}\partial_{k+1}$. We finally define the degree-$k$ homology group of $K$ as $H_k(K)=\ker\partial_k/\operatorname{im}\partial_{k+1}$. We say that two $k$-cycles are \emph{homologous} if they belong to the same homology class, namely when one can be formed from the other by adding any $k$-boundary.

Intuitively, homology group of degree $0,1,2$ capture the gaps between components, tunnels and voids of subsets of $\mathbb{R}^3$. Each generator of a $k$-dimensional homology group can be represented with a $k$-cycle. In practice, this information can be used to visualize the geometry of each hole present in a dataset.

% Refer to a picture here? add context and focus on the important things;

\myheader{Persistent homology.}
%A filter on a simplicial complex $K$ is a function $f:K\to \mathbb{R}$ such that $f^{-1}((-\infty, a])$ is a simplicial complex for all $a\in\mathbb{R}$ and $f^{-1}((-\infty,a])\subset f^{-1}((-\infty,b])$ for all $a<b$ in $\mathbb{R}$. The filter $f$ is a function that determines when a simplex $\sigma$ in introduced to the simplicial complex. In practice, for $K=\{\sigma_1,\dots,\sigma_n\}$, we can impose an order on the simplices on $K$ determined by when they are added to the simplicial complex. With this ordering on the simplices, we define a filtration $F$ as a sequence $\varnothing=K_0\subset K_1\subset K_2\subset \cdots\subset K_n=K$ where $K_i$ is formed from $K_{i-1}$ by adding the simplex $\sigma_i$ to $K_{i-1}$. A common example of a filtration is the Vietoris-Rips complex.
A filtration of a simplicial complex $K$ is a sequence of nested simplicial complexes $\emptyset = K_0 \subset K_1 \subset\cdots\subset K_n=K$. For simplicity, we assume $K_i$ is formed by adding a simplex $\sigma_i$ to $K_{i-1}$. %From this filtration, we can write $K=\{\sigma_1,\dots,\sigma_n\}$ ordered by index of inclusion into the simplicial complex.
Upon adding $\sigma_i$ to $K_{i-1}$, there are only two possible effects: $\sigma_{i}$ either creates a new homology class or destroys an existing one. Depending on the case, we call this simplex a positive simplex or a negative one. For example, adding a 1-simplex (edge) can either connect two existing connected components, or create a new 1-dimensional cycle. Every negative simplex $\sigma_j$ can be paired with a corresponding positive simplex $\sigma_i$ with $i<j$. We call this pairing $(i,j)$ an index persistence pair. If $\sigma_i$ is a positive simplex with no corresponding negative simplex, then $\sigma_i$ creates an \textit{essential homology class}, namely a homology class of $K$. In this work we are particularly interested in a computing representatives of non-essential homology classes, namely the ones which are eventually destroyed by a negative simplex.

\myheader{Boundary matrix.} The boundary matrix of the filtration $F$ with $n$ simplices, is a $n\times n$ binary matrix $M$ where $M_{i,j}=1$ for every pair $(\sigma_i,\sigma_j)$ such that $\sigma_i$ is a proper face of $\sigma_j$, and $0$ otherwise. The order of the columns and rows of $M$ is induced by the order of the simplices in the filtration. From the boundary matrix, we get the index persistence pairs and representatives of persistent homology generators -- which in particular allow us to visualize the changing topology of the filtration. We overview related work in the next section, and provide more details in Section~\ref{sec:algs}.

Reversing the face relationship, we can talk about cofaces, cochains, cocycles, coboundaries, cohomology, and persistent cohomology. We mostly suppress these definitions, and focus on the coboundary matrix, whose columns store proper cofaces of each simplex. We will exploit certain properties of coboundary matrices proven in~\cite{deSilva_2011} to compute homological information. 

%In defining the cohomology of a simplicial complex, the relation of simplices is reversed.Combining with persistence, we have persistent cohomology. We have that the cohomology groups and homology groups at each stage of our filtration remain isomorphic.

\section{Related work}
\label{sec:related}
In this section we briefly review literature on computations of (1-parameter) persistent homology. We put emphasise on work which directly inspired this paper.

The standard boundary matrix reduction algorithm for persistent homology is due to Edelsbrunner, Letscher and Zomorodian~\cite{edelsbrunner2000topological}. The work of de Silva, Morozov and Vejdemo-Johansson~\cite{deSilva_2011} proved various dualities between persistent homology and cohomology, which resulted in increased efficiency of the Dionysus library. Chen and Kerber introduced an important \emph{killing} (clearing) optimization. It played a crucial role in the implementation introduced in the PHAT library~\cite{phat}. This work also experimentally showed the importance of using the clearing optimization along with cohomological computations in the context of skeleta of simplicial complexes. This phenomenon was described in much more detail by Bauer~\cite{bauer2021ripser}, contributing to the efficiency of his popular Ripser package.

Concurrently, various important optimizations~\cite{boissonnat2014simplex, clemont_annotation, boissonnat2020edge, boissonnat2021strong} were developed in the context of the extensive GUDHI library~\cite{maria2014gudhi}.

Computing \emph{optimal} representatives of homology generators is a computationally hard problem in general~\cite{chen2011hardness}, but some special cases are more tractable~\cite{tamal_generators, li2021minimal}. In this work we do not aim at optimality.

It is worth noting that the algorithmic improvements increase the efficiency by a factor of several thousands times compared to the original algorithm (on the same hardware)~\cite{phat}. There are also practical situations in which even the optimized algorithms exhibit prohibitively slow scaling, but switching to coboundary matrices and using the twist technique results in linear scaling and overall fast computation. This is an important motivation for our work, and we elaborate on such situations in Section~\ref{sec:doubletwist}.

This efficient behaviour for datasets arising in practice should be contrasted with the worst-case computational complexity, which is $\Theta(n^3)$, where $n$ is the number of simplices in the input filtration. This bound is realized for a synthetic dataset~\cite{morozov2005persistence}. There exist subcubic algorithms~\cite{milosavljevic2011zigzag}, but they remain of theoretical interest. The result of Edelsbrunner and Parsa~\cite{edelsbrunner2014computational} shows that computing persistence diagrams is as hard as computing the rank of a matrix, leaving little hope for algorithms that would be efficient in the worst case.

We also stress that persistent cohomology and its generating cocycles are powerful tools in their own right~\cite{de2009persistent, scoccola2022toroidal}. In this paper, however, we reduce their role to  computing homology generators -- which in some situations are more natural but harder to compute.

Similar techniques to the ones presented in this work were independently devised by Virk and \v{C}ufar~\cite{ziga}. Combining with techniques from Ripser~\cite{bauer2021ripser} allowed for efficient computation of cycle representatives of Vietoris--Rips complexes in \v{C}ufar's \emph{Ripserer} software.

\section{Existing algorithms}
\label{sec:algs}
We review selected algorithms for persistent homology computations using matrix reduction techniques. Input to these algorithms is a boundary matrix, typically arising from a simplicial filtration. Output is a reduced matrix, from which persistence pairs and representative cycles (of non-essential classes) can be directly obtained.

More precisely, let $M$ be the boundary matrix of a filtration with the $j^{th}$ column of $M$ denoted $M_j$. We define the \textit{lowest-one} of the column $M_j$ as $\text{low}(M_j)=\max\{i=1,\dots,n | M_{i,j}=1\}$, namely the index of the (visually) lowest nonzero entry in the column $M_j$.  % If $M_j$ is reduced to 0, then the boundary of a chain is 0 and thus the chain itself is a cycle. % mention

\myheader{Standard reduction.} The standard matrix reduction algorithm by Edelsbrunner, Letscher and Zomorodian~\cite{edelsbrunner2000topological} can be summarized as a column-wise Gaussian elimination. The goal is to bring the matrix to a reduced form in which each column has a unique lowest one (or is empty). To this end we perform left-to-right column additions, namely $M_{j}\gets M_j+M_i$ for $i < j$ when $\text{low}(M_i)=\text{low}(M_j)$. Due to our choice of coefficients, this removes the entry at index $\text{low}(M_j)$ from the column, necessarily decreasing the value of $\text{low}(M_j)$. See Algorithm~\ref{alg:claspers} for pseudocode, and Figure~\ref{fig:boundaryReduc} for a simple computational example.

We mention some well-known properties of reduced boundary matrices~\cite{edelsbrunner2010computational}.

\begin{property}
\label{prop:index_pers}
Given a reduced matrix $M$, we extract the index persistence pairs as $(\text{low}(M_i), i)$ for each nonzero column $M_i$.     
\end{property}

\begin{property}
\label{prop:zero_columns}
Each zero column in a reduced boundary matrix identifies a positive simplex.
\end{property}

\begin{property}
\label{prop:nonzero_columns}
Each nonzero column of the reduced boundary matrix is a representative of a unique non-essential persistent homology class.    
\end{property}

%From a reduced $M$, we extract persistence pairs $(\sigma_i,\sigma_j)$ which represents the birth and death of a particular homological feature by viewing the pair $(f(\sigma_i),f(\sigma_j))$, where we call $f(\sigma_i)$ the birth of the feature and $f(\sigma_j)$ the death. In practice, from the reduced matrix $M$, we extract $(i,j)$ called the index persistence pair. Similarly, we can track the essential homology classes by tracking positive simplices with no negative pairing, which have the corresponding intervals $(\sigma_{e},\infty)$, with birth being $f(\sigma_e)$ and death being $\infty$. We call the multiset of these pairs the persistence diagram. Intuitively, the essential simplices track when a permanent homological feature appears and the persistent pairs track the lifespan of a particular feature. 

\begin{algorithm}
\caption{Standard Boundary Matrix Reduction}
\label{alg:claspers}
\begin{algorithmic}[1]

\Require Boundary matrix $M$ of a simplicial complex filtration with $n$ columns
\Ensure Reduced boundary matrix $M$
\State $L\gets [0,\dots,0]$ of size $n$
\For{$j=1,\dots,n$}
\While{$M_j\neq 0$ \textbf{and} $L[\text{low}(M_j)]\neq0$}
    \State $M_{j}\gets M_j+M_{L[\text{low}(M_j)]}$
\EndWhile
\If{$M_j\neq0$}
    \State $L[\text{low}(M_j)] \gets j$
\EndIf
\EndFor
\end{algorithmic}
\end{algorithm}

%Computation of persistent cohomology is similar, replacing the boundary matrix $M$ by its anti-transpose, obtained by swapping $M_{i,j}$ with $M_{n+1-j,n+1-i}$. Reducing this new matrix gives the same persistence pairs up to reindexing and thus the same persistence diagram.

\begin{figure}[ht!]
    \centering
    \includegraphics[width=.15\textwidth]{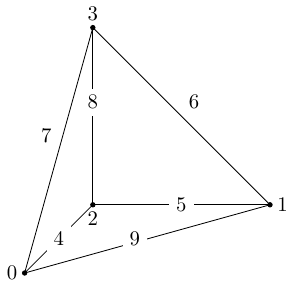}
    \includegraphics[width=.30\textwidth]{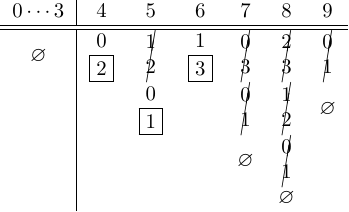}
    \caption{Standard boundary matrix reduction of a filtration of the 1-skeleton of a tetrahedron. The order in which the vertices and edges are added is determined by the numbers. Intermediate state of each column during reduction is shown.}
    \label{fig:boundaryReduc}   
\end{figure}

\myheader{Twist algorithm.}
We recall that each simplex in a filtration either creates a single homology class, or destroys one. If it creates a homology class, the corresponding column in the boundary matrix is necessarily zero (i.e. empty when viewed as a set). This fact was exploited by Chen and Kerber~\cite{Chen2011twist} to develop the \emph{killing technique}, which zeroes the columns corresponding to positive simplices. Their \emph{twist algorithm} for persistence homology capitalizes on this technique by using the reduced columns corresponding to $(p+1)$-simplices to kill columns corresponding to $p$-dimensional ones. To this end the algorithm visits columns in \emph{decreasing} order of dimension.

\myheader{Optimized implementations.}
Algorithm~\ref{alg:opt} outlines an implementation of boundary matrix reduction taking into account useful optimizations. The killing optimization~\cite{Chen2011twist} is employed in line~\ref{line:clearing}. As shown in~\cite{phat}, it is important to decouple the storage of the columns from the storage and handling of the column being currently reduced. In practice, each column is stored as an array of nonzero indices, but the currently reduced column is represented a data-structure which allows for quick updates and maximum queries. The specialized bit-tree data structure described in~\cite{phat} is a good choice. Such data structures  can be costly to initialized, but this is done only once, in line \ref{line:bt1}. It is subsequently used in lines \ref{line:bt2}, \ref{line:bt3}, \ref{line:bt4} at which point it is efficiently cleared.

\begin{algorithm}
\caption{Optimized Boundary Matrix Reduction}
\label{alg:opt}
\begin{algorithmic}[1]
%\Procedure{ReductionWithTwist}
\Require Boundary matrix $M$ of a simplicial complex filtration of dimension $d$ with $n$ columns
\Ensure Reduced boundary matrix $M$
\State $L \gets [0, \dots, 0]$ of size $n$
\State $C \gets$ data-structure for a column of size $n$ \label{line:bt1}
\For {$\delta = d, \dots, 1$}
    \For {$j = 1, \dots, n$}
        \If {$j$ corresponds to simplex of dim $\neq \delta$}
            \State continue
        \EndIf        
        \State copy $M_j$ to $C$ \label{line:bt2}
        \While {$C \neq 0$ and $L[\text{low}(C)] \neq 0$}
            \State  add $M_{L[\text{low}(C)]}$ to $C$ \label{line:bt3} 
        \EndWhile
        \If {$C \neq 0$}
            \State $L[\text{low}(C)] \gets j$ 
            \State $M_{\text{low}(C)} \gets 0$ \label{line:clearing}
        \EndIf
        \State move $C$ to $M_{j}$ \label{line:bt4}    
    \EndFor
\EndFor
\end{algorithmic}
\end{algorithm}

\myheader{Coboundary matrix reduction.}
Persistent homology can be also obtained form a reduced coboundary matrix~\cite{deSilva_2011}. More concretely, we consider the anti-transpose of a boundary matrix obtained by mapping each $0$-based index $i$ to a dual index $i^* = n-1-i$, where $n$ is the total number of simplices. This dual index notation makes it easier to follow computational examples. We remark that in practice it is typically more convenient and faster to generate a coboundary matrix directly rather than by anti-transposing a previously generated boundary matrix.

\begin{property}[Pairing]
\label{prop:copairing}
If $j^*$ is the lowest one of a column $i^*$ in a reduced coboundary matrix, we obtain an index persistence pair $(i^*, j^*)$~\cite{phat, deSilva_2011}. 
\end{property}
We note this is reversed compared to the boundary matrix case: in this case $j$ being the lowest one of $i$ yields $(j, i)$.

The killing technique can be adapted to the case of coboundary matrix as follows~\cite{phat}. The lowest-one $j^*$ of a nonempty column $i^*$ in the coboundary matrix informs us that column $j^*$ of the coboundary matrix can be killed. In this case, however, simplex $\sigma_{j^*}$ is of \emph{higher} dimension, which means we should proceed in an \emph{increasing} order of simplex dimensions. 

\section{Double twist strategy}
\label{sec:doubletwist}
In this section we first explain certain computational issues arising in practice. We then propose a high-level algorithmic strategy which allows to efficiently compute representatives of persistent homology generators using existing software implementations. To this end, we introduce the counterpart of the \emph{killing} technique that we call the \emph{saving} technique and a two-pass algorithm that we call a \emph{double twist} algorithm.

Simplicial filtrations often arise from Vietoris--Rips or {\v C}ech complexes built from a point cloud with $n$ points in $\mathbb{R}^D$. The number of $k$-simplices can be as large as ${n \choose {k+1}}$. Due to this, we are often restricted to the $d$-dimensional skeleton, namely the simplices with dimension not exceeding $d < D$. In practice, $d$ is $2$, $3$, or another small constant. In such cases there are $\Theta(n^{d+1})$ top-dimensional simplices,  and these simplices are the most numerous.

The good news is that persistent homology up to degree $d-1$ can be be computed from such a $d$ dimensional skeleton. There is however a subtle computational downside to restricting the complex in this way. In short, the killing technique cannot be used to zero out the columns corresponding to the top-dimensional simplices (since there are no higher dimensional simplices to do the work). In practice, the bulk of work is spent reducing these columns, which makes the killing technique ineffective and often results in prohibitively slow performance~\cite{phat}. This is not only because these columns are the most numerous -- they are also often harder to reduce, since most of them have to be reduced to zero.

Switching to coboundary reduction largely alleviates this problem, as elaborated in~\cite{deSilva_2011, phat, bauer2021ripser}. Briefly, the coboundary of each top-dimensional simplex in the input skeleton is empty, so no work is needed. Additionally, the killing technique helps zero-out the columns of penultimate dimension.

However, unlike the columns of the reduced boundary matrix (using column operations), the columns of the reduced coboundary matrix (using column operations) do not contain cycle representatives of homology groups. Instead they capture information about cocycle representatives of cohomology groups. We therefore propose a strategy which allows us to obtain the desired homological information while exploiting the improvement granted by coboundary matrix reduction.

We stress that it is possible to obtain representatives of homological generators directly from the coboundary matrix reduction. This however requires tracking the change of basis matrix~\cite{deSilva_2011}, which may incur a significant additional cost. Many software implementations, including the PHAT library, avoid this extra cost; the Eirene software by Henselman-Petrusek is a notable exception. One advantage of our technique is that it allows us to use existing, optimized matrix reduction software without any modification. This is not only easier for the user, but also ensures that performance is not inadvertently degraded due to modifications.

\myheader{Saving technique.}
The new trick is to use the information contained in the reduced coboundary matrix to prune the boundary matrix. Reducing this pruned boundary matrix yields the desired representatives of homological generators. This technique is used in Algorithm~\ref{alg:doubletwist} which we call the \emph{double twist} algorithm. Despite involving two passes, experiments in the next section show that this approach is much faster than reducing the original boundary matrix.

More precisely, if a column $j^*$ is the lowest one of a column in the reduced coboundary matrix, we \emph{save} the corresponding simplex $j = n-1-j^*$. This is done in lines \ref{line:from_save}-\ref{line:to_save} of Algorithm~\ref{alg:doubletwist}. We then construct a boundary matrix -- but only the columns corresponding to the previously saved simplices are generated. Dually, all the remaining columns are set to 0 in the constructed boundary matrix. We emphasize that the boundary matrix is not obtained by anti-transposing the reduced coboundary matrix and zeroing out selected columns, which would lose useful information about cycle representatives. See Figure~\ref{fig:doubletwist} for a computational example using the same input as Figure~\ref{fig:boundaryReduc}. Next, we prove that the proposed algorithm yields correct results.

\begin{prop}[Saving Works]
\label{lem:saving}
The output of the double twist algorithm applied to a filtration $F$ coincides with the reduced matrix $B'$ output by the twist algorithm applied to the boundary matrix $B$ of $F$.
\end{prop}
\begin{proof}
First, Property~\ref{prop:copairing} implies that each saved simplex corresponds to a negative simplex. The remaining ones are therefore positive simplices. Property~\ref{prop:zero_columns} implies that these columns of $B$ would reduce to zero, so the zero columns are returned by both algorithms. Second, the twist algorithm only adds fully reduced columns to other columns on their right. This implies that the columns corresponding to positive simplices do not affect any of the remaining columns of the reduced matrix. Therefore, the nonzero entries of $B$ are reduced in the same way by both algorithms and the outputs coincide.
\end{proof}
Along with Property~\ref{prop:nonzero_columns}, the above implies that $B'$ contains representatives of non-essential persistent homology classes of $F$ as columns.

\myheader{Technicalities.}
We remark that only a single boundary (or coboundary) matrix is stored at a time. After the first pass, the reduced coboundary matrix can be removed from memory, and we only need to store the information about the saved simplices, which is of negligible size. 

\begin{table*}[t!]{
\caption{Results of computational experiments. Each rows corresponds to different datasets. In particular, the two rightmost columns show the timings of the two matrix reductions in our double twist approach.}
\label{tab:results}
\resizebox{0.95\textwidth}{!}
{
\begin{tabular}{|r|c|r|r|r|r||r|r|}
\hline
\thead{Data\\size} & \thead{Maximum\\distance} & \thead{Number of\\simplices} & \thead{Nonzero entries \\ in input (co)boundary matrix} & \thead{Nonzero entries after\\1st pass} & \thead{Boundary matrix \\ reduction  time \\ (naive approach)} & \thead{Coboundary\\reduction time \\ (1st pass of proposed alg.)} & \thead{Pruned boundary\\reduction time \\ (2nd pass of proposed alg.)}\\ \hline
%Old Laptops
%300 & 1.5 & 1,301,252 & 3,873,925 & 86,494 & 11.257s & 0.171s & 0.016s \\ \hline
%300 & 2 & 4,500,250 & 13,455,000 & 134,251 & 55.704s & 0.555s & 0.037s \\ \hline
%400 & 1.5 & 3,019,494 & 9,006,144 & 153,015 & 32.657s & 0.371s & 0.031s \\ \hline
%400 & 2 & 10,667,000 & 31,920,000 & 239,001 & 180.725s & 1.302s & 0.084s \\ \hline
%500 & 1.5 & 5,914,498 & 17,661,937 & 239,672 & 83.529s & 0.720s & 0.055s \\ \hline
%500 & 2 & 20,833,750 & 62,375,000 & 373,751 & 494.723s & 3.373s & 0.155s \\ \hline
%800 & 2 & 85,334,000 & 255,680,000 & 958,001 & 4675.950s & 20.146s & 1.635s \\ \hline
% New Laptop
800 & 1.5 & 24,134,214 & 72,195,145 & 614,492 & 443.613s & 2.445s & 0.058s \\ \hline
800 & 2 & 85,334,000 & 255,680,000 & 958,001 & 2,292.590s & 9.361s & 0.157s \\ \hline
1,600 & 1.3 & 34,482,186 & 103,026,427 & 1, 244,394 & 950.334s & 4.963s & 0.125s \\ \hline
1,600 & 1.5 & 193,843,549 & 580,703,170 & 2,466,432 & 8,160.460s & 23.611s & 0.414s \\ \hline
3,200 & 1 & 5,767,306 & 16,991,750 & 898,505 & 126.506s & 1.391s & 0.043s \\ \hline
3,200 & 1.2 & 89,682,378 & 268,002,505 & 3,101,888 & 4,753.800s & 14.811s & 0.279s \\ \hline
\hline
6,400 & 1 & 44,724,301 & 132,953,878 & 3,593,076 & 2,665.040s & 13.482s & 0.234s \\ \hline
12,800 & 1 & 354,104,851 & 1,057,473,077 & 14,396,429 & $>$7,200s & 205.768s & 2.030s \\ \hline
\end{tabular}
}
}
\end{table*}

\begin{algorithm}
\caption{Double twist strategy}
\label{alg:doubletwist}
\begin{algorithmic}[1]
\Require Filtration $F$ of a simplicial complex
\Ensure Reduced boundary matrix $B$ containing cycle representatives as columns
\State saved-simplices $\gets$ [False$,\dots,$False] of size $n$
%\State generate\_simplices
\State $M \gets $ coboundary matrix of $F$
\State reduce $M$ (Alg.~\ref{alg:opt} but with reversed dimensions)
\For{$i=1,\dots n$} \label{line:from_save}
    \If{$M_i\neq 0$}
        \State $j \gets low(M_i)$
        \State saved-simplices[$n-1-j$] $\gets$ True
    \EndIf
\EndFor \label{line:to_save}
\State delete $M$ from memory
\State $B \gets$ empty boundary matrix with $n$ columns
\For{$i=1,\dots, n$} % generate boundary from saved simplices
    \If{$\text{saved-simplices}[i]=$True}
        \State $B[i] \gets$  boundary of simplex $i$        
    \EndIf    
\EndFor
\State reduce $B$
\end{algorithmic}
\end{algorithm}

\begin{figure}
    \centering
    \includegraphics[width=.3\textwidth]{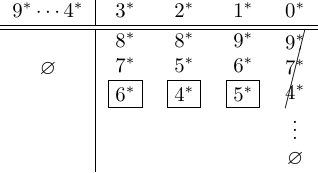}
    \includegraphics[width=.3\textwidth]{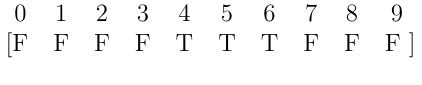}\hfill
    \includegraphics[width=.3\textwidth]{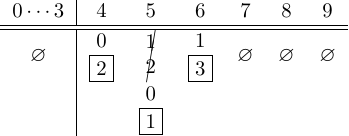}    
    \caption{Double twist applied to the same filtration as in Figure~\ref{fig:boundaryReduc}. After the first pass (top), simplices $4,5,6$ are \emph{saved} as indicated by the array (middle) representing the saved-simplices variable. After the second pass (bottom) we obtain a matrix identical to the reduced matrix in Figure~\ref{fig:boundaryReduc}.}
    \label{fig:doubletwist}
\end{figure}

\section{Experiments}
\label{sec:experiments}

In our experiments we focus on Vietoris--Rips filtrations coming from synthetic data.
The aim is to check how the two matrix reductions in the double-twist algorithm scale, compared to a naive strategy in which the boundary matrix is reduced directly. We note that we use a fully optimized twist reduction also for the naive strategy. We also aim to verify that the boundary matrix used in the second pass of our algorithm contains only a small number of nonzero elements.

We implemented our strategy in C++ using the PHAT library. We used a {Clang 14.0.3} compiler. The experiments were done on a single core of a 3.5 GHz CPU with 32 GB RAM.

The results are presented in Table~\ref{tab:results}. Each row corresponds to a single dataset. Each dataset is samples of a 9-dimensional sphere in $\mathbb{R}^{10}$ with a different number of points and radius cutoff. We benchmark both strategies on the 2-skeleton of the Vietoris--Rips filtration for each dataset.

\myheader{Observations.} Analyzing Table~\ref{tab:results}, we observe that the second pass of the double twist algorithm has negligible impact on the overall execution time (less than 10\%). This is not surprising, given the number of nonzero elements in pruned boundary matrix is at least an order of magnitude smaller than compared to the original (co)boundary matrix. Overall, the proposed two-pass algorithm is much faster (up to 200 times) than the naive approach. 

%We remark that there are large practical examples~\cite{phat} which would show an even bigger performance gap between the two approaches -- except the naive strategy would take prohibitively long time to finish. 

The experiments clearly show that there are situations in which using the new double twist strategy is beneficial compared to direct reduction of the boundary matrix. Finally, we add that the results apply also to other situations in which the top-dimensional cells dominate, for example low-dimensional skeleta of high-dimensional \v{C}ech, Delaunay or even cubical filtrations.

\section{Summary}
\label{sec:summary}
We proposed a simple algorithmic strategy and showed that -- in certain important situations -- it significantly speeds up computation of representatives of persistent homology generators. We stress that it does not require implementing a new matrix-reduction algorithm. Instead, any optimized implementation can be used in a black-box fashion, provided it computes index persistence pairs.

We also reiterate that only the representatives of non-essential classes can be obtained from the reduced boundary matrix. While this is actually the more interesting information captured by the persistent homology, we hope that future software packages will support efficient implementation of all representatives. In the meanwhile, however, techniques like the one we proposed serve as a useful workaround.

\bibliographystyle{plainurl}
\bibliography{main}

\end{document}